\documentclass[reqno, final,a4paper]{amsart}
\usepackage{color}
\usepackage[colorlinks=true,allcolors=blue
]{hyperref}
\usepackage{amsmath, amssymb, amsthm}
\usepackage{mathrsfs}
\usepackage{mathtools}
\usepackage[noabbrev,capitalize,nameinlink]{cleveref}
\crefname{equation}{}{}
\usepackage{fullpage}
\usepackage[noadjust]{cite}
\usepackage{graphics}
\usepackage{pifont}
\usepackage{tikz}
\usetikzlibrary{calc}
\usepackage{bbm}
\usepackage[T1]{fontenc}
\usetikzlibrary{arrows.meta}

\usepackage{environ}
\usepackage{framed}
\usepackage{url}
\usepackage[linesnumbered,ruled,vlined]{algorithm2e}
\usepackage[noend]{algpseudocode}
\usepackage[labelfont=bf]{caption}
\usepackage{framed}
\usepackage[framemethod=tikz]{mdframed}
\usepackage{pgfplots}
\usepackage{appendix}
\usepackage{graphicx}
\usepackage[textsize=tiny]{todonotes}
\usepackage{tcolorbox}
\usepackage{xcolor}
\usepackage[shortlabels]{enumitem}
\usepackage{colonequals}
\crefformat{enumi}{#2#1#3}
\crefrangeformat{enumi}{#3#1#4 to~#5#2#6}
\crefmultiformat{enumi}{#2#1#3}
{ and~#2#1#3}{, #2#1#3}{ and~#2#1#3}
\allowdisplaybreaks

\let\originalleft\left
\let\originalright\right
\renewcommand{\left}{\mathopen{}\mathclose\bgroup\originalleft}
\renewcommand{\right}{\aftergroup\egroup\originalright}

\newenvironment{poc}{\begin{proof}[Proof of claim]}{\end{proof}}

\crefname{algocf}{Algorithm}{Algorithms}

\crefname{equation}{}{} 
\AtBeginEnvironment{appendices}{\crefalias{section}{appendix}} 

\usepackage[color]{showkeys} 

\colorlet{refkey}{orange!20}
\colorlet{labelkey}{blue!30}

\crefname{algocf}{Algorithm}{Algorithms}

\numberwithin{equation}{section}
\newtheorem{theorem}{Theorem}[section]
\newtheorem{proposition}[theorem]{Proposition}
\newtheorem{lemma}[theorem]{Lemma}
\newtheorem{claim}[theorem]{Claim}
\crefname{claim}{Claim}{Claims}
\crefname{theorem}{Theorem}{Theorems}

\newtheorem{corollary}[theorem]{Corollary}

\newtheorem*{claim*}{Claim}
\newtheorem{fact}[theorem]{Fact}
\crefname{conjecture}{Conjecture}{Conjectures}

\newtheorem{definition}[theorem]{Definition}

\newtheorem*{definition*}{Definition}

\theoremstyle{definition}
\theoremstyle{remark}
\newtheorem{remark}[theorem]{Remark}

\newcommand{\mb}{\mathbb}
\newcommand{\mbf}{\mathbf}

\newcommand{\mc}{\mathcal}

\newcommand{\on}{\operatorname}

\newcommand{\eps}{\varepsilon}

\allowdisplaybreaks
\pgfplotsset{compat=1.18}

\begin{document}

\title{Relative discrepancy of hypergraphs}

\author[Luong-Le]{Diep Luong-Le}
\address{Vingroup Big Data Institute, 9th floor, Century Tower, Times City, 458 Minh Khai, Vinh Tuy Ward, Hai Ba Trung District, Hanoi, Vietnam.}
\email{dnl2128@columbia.edu}

\author[Tran]{Tuan Tran}
\author[Yang]{Dilong Yang}

\address{School of Mathematical Sciences, University of Science and Technology of China, Hefei, Anhui, China.}
\email{trantuan@ustc.edu.cn}
\email{ydl2001@mail.ustc.edu.cn}

\begin{abstract}
Given $k$-uniform hypergraphs $G$ and $H$ on $n$ vertices with densities $p$ and $q$, their relative discrepancy is defined as $\on{disc}(G,H)=\max\big||E(G')\cap E(H')|-pq\binom{n}{k}\big|$, where the maximum ranges over all pairs $G',H'$ with $G'\cong G$, $H'\cong H$, and $V(G')=V(H')$. Let $\on{bs}(k)$ denote the smallest integer $m \ge 2$ such that any collection of $m$ $k$-uniform hypergraphs on $n$ vertices with moderate densities contains a pair $G,H$ for which $\on{disc}(G,H) = \Omega(n^{(k+1)/2})$.

In this paper, we answer several questions raised by Bollob\'as and Scott, providing both upper and lower bounds for $\on{bs}(k)$.
Consequently, we determine the exact value of $\on{bs}(k)$ for $2\le k\le 13$, and
show $\on{bs}(k)=O(k^{0.525})$, substantially improving the previous bound $\on{bs}(k)\le k+1$ due to Bollob\'as--Scott. The case $k=2$ recovers a result of Bollob\'as--Scott, which generalises classical theorems of Erd\H{o}s--Spencer, and Erd\H{o}s--Goldberg--Pach--Spencer. The case $k=3$ also follows from the results of Bollob\'as--Scott and Kwan--Sudakov--Tran. Our proof combines linear algebra, Fourier analysis, and extremal hypergraph theory. 
\end{abstract}
\maketitle

\section{Introduction}

Discrepancy theory, which traces its roots to Weyl’s early 20th-century work on equidistribution, investigates how far actual configurations deviate from expected patterns within various mathematical structures. Since its inception, it has driven significant advances across a range of disciplines, including measure theory, number theory, computational geometry, optimisation, and theoretical computer science. For a comprehensive overview, we refer the reader to Chazelle’s monograph \cite{Cha00}.

In the discrete setting, Erdős and Spencer \cite{ES71} introduce the notion of {\it hypergraph discrepancy} to quantify deviations in the number of edges in induced subhypergraphs from their expected values. For a $k$-uniform hypergraph $G$ on $n$ vertices with edge density $p=e(G)/\binom{n}{k}$,
its discrepancy is defined as
\[
\on{disc}(G)=\max_{U\subseteq V(G)}\Big| e(G[U])-p\binom{|U|}{k} \Big|.
\]	
They show that when $p=1/2$, the discrepancy satisfies $\on{disc}(G)=\Omega(n^{(k+1)/2})$, which is tight for random $k$-uniform hypergraphs with this density. For graphs, this result is extended to all fixed densities $p \in (0,1)$ by Erd\H{o}s, Goldberg, Pach, and Spencer \cite{EGPS88}, and is further generalised to arbitrary densities for $k \geq 3$ by Bollob\'as and Scott \cite{BS06}. For recent advances in the sparse regime $p = o(1)$, we refer the reader to \cite{RST23,RT24}.

Building on these foundations, Bollob\'as and Scott \cite{BS11} introduce the {\it relative discrepancy of two hypergraphs}, which measures the extent to which the edge sets are independently and uniformly distributed. Given $k$-uniform hypergraphs $G$ and $H$, both on $n$ vertices and with densities $p$ and $q$, their relative discrepancy is defined as
\[
\on{disc}(G,H)=\max\Big||E(G')\cap E(H')|-pq\binom{n}{k}\Big|, 
\]
where the maximum ranges over all pairs $G',H'$ with $G'\cong G$, $H'\cong H$, and $V(G')=V(H')$. 

Bollob\'as and Scott \cite{BS14}, and Ma, Navev, and Sudakov \cite{MNS14} determine the asymptotic of $\on{disc}(G,H)$ when $G$ and $H$ are two random hypergraphs. By selecting $H$ from specific families, one obtains refined measures of the edge distribution of $G$. For example, choosing $H=K_{\lfloor n/2\rfloor,\lceil n/2\rceil}$ recovers the bipartite discrepancy introduced in \cite{EGPS88}. Another notable case arises when $H=H_m$ is the disjoint union of a complete $k$-uniform hypergraph on $m$ vertices and $n-m$ isolated vertices. We have $\on{disc}(G,H_m)=\max_{|U|=m}\big|e(G[U])-p\binom{|U|}{k}\big|$, so that
\begin{equation}\label{eq:pair-to-single}
\on{disc}(G)=\max_{m}\on{disc}(G,H_m).    
\end{equation}

Motivated by the result of Erd\H{o}s, Goldberg, Pach, and Spencer, Bollob\'as and Scott \cite{BS11} conjecture that any pair of $k$-uniform hypergraphs with moderate densities must exhibit large discrepancy. If true, this conjecture would generalise all aforementioned results via formula~\eqref{eq:pair-to-single}. The conjecture is confirmed for graphs \cite{BS06}, but a counterexample is later found for $k=3$ in \cite{BS15}. 

In contrast, the picture for weighted hypergraphs is well-understood. Bollob\'as and Scott \cite{BS15} show that for every $k \geq 1$, any set of $k+1$ nontrivial normalised weighted $k$-uniform hypergraphs contains a pair with relative discrepancy at least $\Omega_k(n^{(k+1)/2})$. Moreover, the number of hypergraphs required is best possible: there exists a set of $k$ nontrivial weighted $k$-uniform hypergraphs in which every pair has discrepancy zero. They ask what happens when one restricts to the (arguably more natural) unweighted case.

This paper answers these questions and reveals that the unweighted setting behaves strikingly differently from the weighted case. For $k\ge 2$, let $\on{bs}(k)$ denote the smallest integer $m\ge 2$ such that any collection $G_1,\ldots,G_m$ of $n$-vertex $k$-uniform hypergraphs with densities in $(\gamma,1-\gamma)$ contains a pair $G_i,G_j$ with $\on{disc}(G_i,G_j)=\Omega_{k,\gamma}(n^{(k+1)/2})$. From the work of Bollob\'as and Scott, it is known that $\on{bs}(2)=2$ and $2\le \on{bs}(k) \le k+1$ for all $k\ge 2$. Our main result provides a good upper bound on $\on{bs}(k)$ in terms of a number-theoretic function. We also prove $\on{bs}(k)\ge 3$ for all $k\ge 3$, thus disproving the conjecture of Bollob\'as and Scott.
Consequently, we obtain $\on{bs}(k)=3$ for all $3\leq k\leq 13$ (already new for $k=3$), and show $\on{bs}(k)=O(k^{0.525})$, substantially improving the previous bound $\on{bs}(k)\le k+1$. 

To present our result, we first introduce a key definition.

\begin{definition}\label{defn:gap}
For an integer $k\ge 2$, let $g(k)$ denote the smallest integer $g\ge 0$ such that for each integer $m\in [2,k]$, the following binary system of linear equations
\[
\sum_{0\le i \le r}(-1)^i\binom{r}{i}\alpha_i=0 \quad \text{for } r=m-g,\ldots,m,
\]
with $\vec{\alpha}=(\alpha_0,\ldots,\alpha_{m})\in \{0,1\}^{m+1}$, has only two solutions $\vec{\alpha}=(0,\ldots,0)$ and $\vec{\alpha}=(1,\ldots,1)$.    
\end{definition} 

Our main result is the following.

\begin{theorem}\label{thm:main}
For any integer $k\ge 2$, 
\[
\min\{k,3\}\le \on{bs}(k)\le g(k)+2.
\]
\end{theorem}
The short proof of the lower bound, presented in \cref{sec:lower-bound}, relies on Keevash's existence theorem for designs \cite{Kee14} as a black box. In contrast, the proof of the upper bound is much more involved (see \cref{sec:outline} for a proof outline).

Interestingly, the quantity $g(k)$ also appears in the work of von zur Gathen and Roche \cite{GR97}, in connection with a question of Mario Szegedy on the Fourier degree of symmetric Boolean functions.
We will make use of the following result from their paper.\footnote{They actually define a function $\Gamma\colon \mb{N}\to \mb{N}\cup\{0\}$, which relates to $g(k)$ via $g(k)=\max_{2\le m\le k}\Gamma(m)$. Part (i) of \cref{thm:Gathen-Roche} is \cite[Table 3]{GR97}.}

\begin{theorem}[\cite{GR97}]\label{thm:Gathen-Roche}\textcolor{white}{}

\begin{itemize}
    \item[\rm (i)] $g(2)=0$, $g(k)=1$ for $3\le k\le 13$, and $g(k)\le 3$ for $14\le k\le 128$.
    \item[\rm (ii)] Let $k\ge 2$, and define $G(k)$ as the maximal gap between an integer $m\in [3,k+1]$ and the largest prime not exceeding $m$. Then, $g(k)\le G(k)$.
\end{itemize}
\end{theorem}

Since part (ii) of \cref{thm:Gathen-Roche} is not explicitly stated in \cite{GR97}, we include a short proof of it in the appendix. A result of Baker, Harman, and Pintz \cite{BHP01} gives $G(k)=O(k^{0.525})$. Combining this with \cref{thm:main,thm:Gathen-Roche}, we obtain:

\begin{corollary}\label{cor:main-result}\textcolor{white}{}
\begin{itemize}
    \item[\rm (i)] $\on{bs}(2)=2$, $\on{bs}(k)=3$ for $3\le k\le 13$, and $3\le \on{bs}(k)\le 5$ for $14\le k\le 128$.
    \item[\rm (ii)] For every $k\in \mb{N}$, $\on{bs}(k)=O(k^{0.525})$.
\end{itemize}    
\end{corollary}

We note that $\on{bs}(3)=3$ also follows from the result of Bollob\'as and Scott \cite[Theorem 3]{BS15}, and of Kwan, Sudakov, and Tran \cite[Lemma 5.1]{KST19}.

\begin{itemize}
\item Based on a probabilistic model of primes, Cram\'er \cite{Cra36} conjectures $\limsup_{k\to \infty}\frac{G(k)}{\log^2 k}=1$. Numerical evidence from prime computations up to $4\cdot 10^{18}$ \cite{EHP14} strongly supports this conjecture: specifically, $\max_{k\le 4\cdot 10^{18}}G(k)/\log^2 k\approx 0.9206$, which is slightly below the value predicted by Cram\'er.
\item  It seems likely that the true upper bound for $g(k)$—and hence for $\on{bs}(k)$—is substantially smaller than $O(\log^2 k)$. For instance, von zur Gathen and Roche \cite{GR97} conjecture $g(k)=O(1)$, which would imply that every nonconstant symmetric Boolean function of $k$ variables has Fourier degree at least $k-\Omega(1)$, and that $\on{bs}(k)=O(1)$.
\end{itemize}

\subsection{Proof overview}\label{sec:outline}
We now outline the proof of \cref{thm:main}. To illustrate the ideas, we introduce the quantity $\on{bs}^*(k)$, which is defined similarly to $\on{bs}(k)$, but replaces $\on{disc}(G_i, G_j) = \Omega(n^{(k+1)/2})$ with the weaker requirement $\on{disc}(G_i, G_j)>0$. Clearly, $\on{bs}^*(k)\le \on{bs}(k)$. To bound $\on{bs}^*(k)$ from above, we use a result of Bollob\'as and Scott \cite{BS15} that relates the relative discrepancy of two hypergraphs to their $W$-vectors. 

Given a $k$-uniform hypergraph $G$ on $n$ vertices, Bollob\'as and Scott define its $W$-vector $(W_1(G), \ldots, W_k(G))$ by
\begin{equation*}
W_r(G)= \frac{1}{n(n-1)\cdots (n-2r+1)\cdot 2^r\binom{n-2r}{k-r}} \cdot \sum_{(a_1,b_1,\ldots,a_r,b_r)} \Big| \sum_{S} (-1)^{|S\cap \{b_1,\ldots,b_r\}|}\deg(S)\Big| \text{\quad for } 1\leq r\leq k,
\end{equation*}
where the outer sum runs over all sequences of $2r$ distinct vertices, and the inner sum runs over all $S\in \binom{V(G)}{r}$ such that $|S\cap \{a_i,b_i\}|=1$ for every $i\in [r]$. 
Here, $\deg(S)$ denotes the number of edges of $G$ containing $S$. Note that $W_r(G)\in [0,1]$ for all $r$.

For any $k$-uniform hypergraphs $G$ and $H$ on the vertex set $[n]$, Bollob\'as and Scott \cite{BS15} prove
\begin{equation*}
\on{disc}(G,H) \ge c_k n^{(k+1)/2}\cdot \max_{1\le r\le k} W_r(G)W_r(H).  
\end{equation*}
Suppose for some $\ell \le k$, every $k$-uniform hypergraph $G$ of moderate density satisfies 
\begin{equation}\label{eq:W-components}
\max_{\ell \le r \le k}W_r(G)>0.    
\end{equation} 
Then, by the pigeonhole principle, 
\[
\on{bs}^*(k) \le k-\ell+2.
\] 
In \cite{BS15}, Bollob\'as and Scott establish \eqref{eq:W-components} with $\ell = 1$, leading to the bound $\on{bs}(k) \le k + 1$. However, it is unclear whether their methods can be extended to prove \cref{eq:W-components} with $\ell = 2$, even for large $k$. On the other hand, to achieve an upper bound of the form $\on{bs}(k)\le O(k^{1-c})$ through \eqref{eq:W-components}, one would need $\ell=k-O(k^{1-c})$.

When $k=3$, Kwan, Sudakov, and Tran \cite{KST19} characterise all moderate-density $3$-uniform hypergraphs $G$ for which $W_3(G)=0$ by using results from extremal hypergraph theory along with a delicate case analysis. This characterisation implies \eqref{eq:W-components} for $k=3$ and $\ell=2$, giving $\on{bs}^*(3)\le 3$. For $k>3$, such a characterisation remains out of reach. 

The key contribution in this paper is a new method to establish \eqref{eq:W-components}. We prove $\on{bs}^*(k) \le k - \ell + 2$ by ruling out the possibility that $W_k(G), \ldots, W_{\ell}(G)$ are all zero. Specifically, we show in \cref{prop:algebraic-criterion} that $\max_{\ell\le r\le k}W_r(G)=0$ if and only if there exists a function $h\colon \binom{V(G)}{\ell-1}\to \mb{R}$ such that for all $R\in \binom{V(G)}{k}$,
\begin{equation*}\label{eq:algebraic-criterion}
\mathbf{1}_{G}(R)=\sum_{S\subset R, \ |S|=\ell-1}h(S),    
\end{equation*}
where $\mathbf{1}_{G}(R)=1$ if $R\in E(G)$ and $0$ otherwise.
This algebraic criterion enables a local-to-global approach. Suppose, for contradiction, that $\max_{\ell\le r\le k}W_r(G)=0$. 
Given that $G$ has moderate density, a result of Fox and Sudakov \cite{FS08} ensures that $G$ contains a large non-homogeneous multipartite pattern $F$ as an induced subhypergraph. (A hypergraph is non-homogeneous if it is neither empty nor complete.) 
By the algebraic criterion, the induced subhypergraph $F\subseteq G$ also satisfies $\max_{\ell\le r\le k}W_r(F)=0$.
To reach a contradiction, we prove by induction on the number of parts of $F$ that $\max_{\ell \le r\le k} W_r(F) > 0$. The base case, where $F$ is a bipartite pattern, naturally introduces the extremal function $g(k)$.

We combine the strategy outlined above with a stability approach to bound $\on{bs}(k)$. In \cref{prop:stability}, we show that if $\max_{\ell \le r \le k} W_r(G) = o(1)$, then there exists a function $\tilde{f} \colon \binom{V(G)}{k} \to \mathbb{R}$ such that 
\begin{itemize}
\item[\rm (i)] $|1_G(R)-\tilde{f}(R)|=o(1)$ for all but at most $o(n^k)$ sets $R\in \binom{V(G)}{k}$, and
\item[\rm (ii)] $\max_{\ell\le r\le k}W_r(\tilde{f})=0$.
\end{itemize}
 Note that the $W$-vectors are not well-suited for establishing this stability result. Instead, we use their $L_2$-based analogue, which facilitates the application of Fourier analysis on the space of multivariate harmonic polynomials. 
 
 Another key ingredient is a robust version of the Fox--Sudakov result (see \cref{prop:robust-unavoidable}), which we prove using Ramsey-type and probabilistic arguments. Now, let $W$ denote the set of $R\in \binom{V(G)}{k}$ for which $|1_G(R)-\tilde{f}(R)|=\Omega(1)$. From property (i), we have $|W|=o(n^k)$. Using the robust Fox-Sudakov result, we find a large non-homogeneous multipartite pattern $F$ as an induced subhypergraph of $G$ that is disjoint from $W$. The algebraic criterion, combined with a continuity argument, then shows $\max_{\ell \le r\le k}W_r(F)=0$. From here, we proceed as before.

It is worth noting that Jain, Kwan, Mubayi, and Tran \cite{JKMT25} use Bollob\'as and Scott's bound, specifically $\max\limits_{1 \leq r \leq k} W_r(G) = \Omega(1)$, among other tools, to investigate anticoncentration of edge-statistics in large hypergraphs. Our approach offers a conceptual proof of this bound (see \cref{cor:BS-L2}).

\subsection{Notation and organisation}
We use standard notation throughout. Let $[n]=\{1,2,\ldots,n\}$, and let $\binom{X}{k}$ denote
the family of all $k$-element subsets of a set $X$. For a $k$-uniform hypergraph $G$, we write $\mathbf{1}_G$ for the indicator function on $\binom{V(G)}{k}$ defined by $\mathbf{1}_G(S) = 1$ if $S \in E(G)$ and $0$ otherwise. If $X$ is a finite set, we write $x \sim X$ to indicate that $x$ is chosen uniformly at random from $X$. For a function $f \colon X \to \mathbb{R}$, we define $\|f\|_{\infty} = \sup_{x \in X} |f(x)|$.

In \cref{sec:W-vectors}, we prove the algebraic criterion and its stability counterpart. While these results are central to the proof of \cref{thm:main}, their application requires a robust extension of the Fox–Sudakov theorem, which we present in \cref{sec:robust-unavoidable}. Finally, in \cref{sec:proofs}, we combine these ingredients with additional combinatorial techniques to complete the proof of \cref{thm:main}.

\subsection*{Acknowledgments} We thank Jaehoon Kim for bringing \cite{GKLO21} to our attention and Hao Huang for helpful discussions related to \cref{lem:pattern}.
Tran was supported by the National Key Research and Development Program of China 2023YFA101020.

\section{\texorpdfstring{The $W$-vectors}{The W-vectors}}\label{sec:W-vectors}
In this section, we define the $W$-vector of a weighted $k$-uniform hypergraph and present its key properties. Given $s\in \mb{N}$ and a set $V$ with $|V|\ge s$, an {\bf $s$-permutation} of $V$ is a sequence of $s$ distinct elements of $V$.

\begin{definition}\label{defn:W-vector}
Let $k,r\in \mb{N}$ and $V$ be a finite set with $|V|\ge 2k \ge 2r$. Given a function $f\colon \binom{V}{k}\to \mb{R}$, the level-$r$ weight of $f$ is defined as
\begin{equation*}\label{eq:W-vector}
W^r(f)=\left(\underset{(a_1,b_1,\ldots,a_r,b_r)}{\mb{E}}\Big(\underset{R}{\mb{E}}\, (-1)^{|R\cap \{b_1,\ldots,b_r\}|}f(R)\Big)^2\right)^{1/2},
\end{equation*}
where $(a_1,b_1,\ldots,a_r,b_r)$ is a uniformly random $2r$-permutation of $V$, and $R$ is a uniformly random $k$-subset of $V$ satisfying $|R\cap \{a_i,b_i\}|=1$ for each $i\in [r]$. We define the level-$r$ weight of a $k$-uniform hypergraph $G$ on the vertex set $V$ as $W^r(G)=W^r(\mathbf{1}_{G})$.
\end{definition}

The following is a direct consequence of the Cauchy-Schwarz inequality, and we state it here for future reference.

\begin{fact}\label{fact}
Let $k,r\in \mb{N}$ and $V$ be a finite set with $|V|\ge 2k \ge 2r$. For any $f,g\colon \binom{V}{k}\to \mb{R}$, 
\[ 
\left|W^r(f)- W^r(g)\right|\le W^r(f-g)\le \| f-g\|_{\infty}.
\]
\end{fact}

We need a few key ingredients to prove \cref{thm:main}. First, we need a lemma essentially due to Bollob\'as and Scott \cite{BS15}.

\begin{lemma}\label{lem:Bollobas-Scott}
Let $G$ and $H$ be $n$-vertex $k$-uniform hypergraphs with $n\ge 2k$. Then
\[
\on{disc}(G,H) \ge c_{\ref{lem:Bollobas-Scott}} n^{(k+1)/2}\cdot \max_{1\le r\le k} (W^r(G)W^r(H))^2,  
\]
where $c_{\ref{lem:Bollobas-Scott}}>0$ depends only on $k$.
\end{lemma}

\begin{proof}
For any function $f\colon \binom{[n]}{k}\to \mb{R}$, Bollob\'as and Scott \cite{BS15} define\footnote{Strictly speaking, they first
define $W_r(f)$ in a slightly different way and then  
prove that the two definitions are equivalent; see \cite[Lemma~8]{BS15}.}
\[
W_r(f)=\underset{(a_1,b_1,\ldots,a_r,b_r)}{\mb{E}}\Big|\underset{R}{\mb{E}}\, (-1)^{|R\cap \{b_1,\ldots,b_r\}|}f(R)\Big|,
\]
where $(a_1,b_1,\ldots,a_r,b_r)$ is a uniformly random $2r$-permutation of $[n]$, and $R$ is a uniformly random $k$-subset of $[n]$ satisfying $|R\cap \{a_i,b_i\}|=1$ for each $i\in [r]$. Suppose further that $\|f\|_{\infty} \le 1$. Then the inner expectation is bounded by 1 in absolute value, and we obtain
\[
W_r(f)\ge \underset{(a_1,b_1,\ldots,a_r,b_r)}{\mb{E}}\Big|\underset{R}{\mb{E}}\, (-1)^{|R\cap \{b_1,\ldots,b_r\}|}f(R)\Big|^2 = (W^r(f))^2.
\]
Applying this to $f=\mathbf{1}_{G}$ and $f=\mathbf{1}_{H}$, we obtain $W_r(G)\ge (W^r(G))^2$ and $W_r(H)\ge (W^r(H))^2$. Moreover, \cite[Theorem 3]{BS15} shows $\on{disc}(G,H) \ge c_{\ref{lem:Bollobas-Scott}} n^{(k+1)/2}\cdot \max_{1\le r\le k} W_r(G)W_r(H)$.
Combining the two bounds yields the desired result.
\end{proof}

The following result characterises the weighted $k$-uniform hypergraphs $f$ for which $\max_{\ell \le r \le k}W^r(f)=0$. We will prove this in \cref{sec:algebraic-criterion} using rank-based arguments.

\begin{proposition}\label{prop:algebraic-criterion}
Let $k, \ell \in \mb{N}$ and let $V$ be a finite set with $|V| \ge 2k \ge 2\ell$. For any $f \colon \binom{V}{k} \to \mb{R}$, the following are equivalent:
\begin{itemize}
    \item[\rm (i)] $W^k(f) = \cdots = W^{\ell}(f) = 0$;
    \item[\rm (ii)] There exists a unique $h \colon \binom{V}{\ell-1} \to \mb{R}$ such that for all $R \in \binom{V}{k}$, we have $f(R) = \sum_{S \subset R} h(S)$.
\end{itemize}
\end{proposition}

\begin{remark}\textcolor{white}{}\label{rmk:uniqueness-induced}
\begin{enumerate}
    \item If $f\equiv 0$, then by the uniqueness of $h$, we must have $h\equiv 0$.
    \item Suppose $G$ is a $k$-uniform hypergraph with $W^k(G)=\cdots=W^{\ell}(G)=0$. Then, for any induced subhypergraph $F\subseteq G$ with at least $2k$ vertices, $W^k(F)=\cdots=W^{\ell}(F)=0$.
\end{enumerate}
\end{remark}

We also require a stability version of \cref{prop:algebraic-criterion}. The proof, which is deferred to \cref{sec:stability}, is based on an explicit orthogonal basis for the space of multilinear harmonic polynomials.

\begin{proposition}\label{prop:stability}
Let $k,\ell\in \mb{N}$, $\eps>0$, and $V$ be a finite set with $|V|\ge 2k \ge 2\ell$. If $f\colon \binom{V}{k}\to \mb{R}$ satisfies $\max_{\ell \le r \le k}W^r(f)\le 
\eps$, then there exists $\tilde{f}\colon \binom{V}{k}\to \mb{R}$ with the following properties:
\begin{itemize}
    \item[\rm (i)] For all but at most $5^{k/2}\eps^{1/2}\binom{n}{k}$ sets $R\in \binom{[n]}{k}$, we have $|f(R)-\tilde{f}(R)| \le \eps^{1/2}$;
    \item[\rm (ii)] $W^k(\tilde{f})=\cdots=W^{\ell}(\tilde{f})=0$.
\end{itemize}   
\end{proposition} 

\subsection{The algebraic criterion}\label{sec:algebraic-criterion}

This section is dedicated to proving \cref{prop:algebraic-criterion}. We begin by deriving it from \cref{lem:base-case,lem:inclusion-matrix}, which
are stated below.

\begin{proof}[Proof of \cref{prop:algebraic-criterion} assuming \cref{lem:inclusion-matrix,lem:base-case}]
Let $X_{k,\ell}$ be the space of functions $f\colon\binom{[n]}{k}\to \mb{R}$ with $W^k(f)=\cdots=W^{\ell}(f)=0$, and let $Y_{k,\ell}$ consist of those $f\colon\binom{[n]}{k}\to \mb{R}$ for which there exists a function  $h\colon \binom{[n]}{\ell-1}\to \mb{R}$ satisfying 
$f(R)=\sum_{S\subset R}h(S)$ for all $R\in \binom{[n]}{k}$. 

We first show $Y_{k,\ell}\subseteq X_{k,\ell}.$ Let $f\in Y_{k,\ell}$ and fix any $r\in \mb{N}$ with $\ell \le r \le k$. Consider any $2r$-permutation $(a_1,b_1,\ldots,a_r,b_r)$, and let $\mc{P}$ be the family of $k$-subsets $R\subset [n]$ such that $|R\cap\{a_i,b_i\}|=1$ for all $i\in [r]$. Consider any set $S\in \binom{[n]}{\ell-1}$. Since $|S|=\ell-1<r$, there exists $i_0\in [r]$ with $S\cap \{a_{i_0},b_{i_0}\}=\emptyset$. The subfamily $\{R\in \mc{P}\colon S\subset R\}$ can be divided into pairs $R,R^*$ differing only in the choice between $a_{i_0}$ and $b_{i_0}$. For each such pair, the values $|R\cap \{b_1,\ldots,b_r\}|$ and $|R^*\cap \{b_1,\ldots,b_r\}|$ differ by exactly one. Therefore, the coefficient of $h(S)$ in the average $\underset{R\sim \mc{P}}{\mb{E}}\, \Big[(-1)^{|R\cap \{b_1,\ldots,b_r\}|}f(R)\Big]$ is
\[
\frac{1}{|\mc{P}|}\sum_{R\in \mc{P}\colon S\subset R}(-1)^{|R\cap \{b_1,\ldots,b_r\}|}=\frac{1}{2|\mc{P}|}\sum_{R,R^*}\left((-1)^{|R\cap \{b_1,\ldots,b_r\}|}+(-1)^{|R^*\cap \{b_1,\ldots,b_r\}|}\right)=0.
\]
Thus, the average vanishes for every $2r$-permutation $(a_1,b_1,\ldots,a_r,b_r)$, implying $W^r(f)=0$. Since $r$ is arbitrary in $[\ell,k]$, we conclude $f\in X_{k,\ell}$, so 
$Y_{k,\ell} \subseteq X_{k,\ell}$.

Next, we prove the reverse inclusion $X_{k,\ell}\subseteq Y_{k,\ell}$. Consider any $f\in X_{k,\ell}$. As $W^k(f)=0$, by \cref{lem:base-case} below, there exists $h_1\colon \binom{[n]}{k-1}\to \mb{R}$ such that $f(R)=\sum_{S_1\subset R}h_1(S_1)$ for all $R\in \binom{[n]}{k}$. Now, by a similar calculation to the one above, we find that for every $2(k-1)$-permutation $(a_1,b_1,\ldots,a_{k-1},b_{k-1})$, 
\[
0=\underset{R}{\mb{E}}\, \Big[(-1)^{|R\cap \{b_1,\ldots,b_{k-1}\}|}f(R)\Big]=\underset{S_1}{\mb{E}}\, \Big[(-1)^{|S_1\cap \{b_1,\ldots,b_{k-1}\}|}h_1(S_1)\Big],
\]
where $R$ and $S_1$ are uniformly random $k$- and $(k-1)$-subsets of $[n]$, respectively, such that
\[
|R \cap \{a_i, b_i\}| = 1 \quad \text{and} \quad |S_1 \cap \{a_i, b_i\}| = 1 \quad \text{for all } i \in [k-1].
\]
This shows $W^{k-1}(h_1)=0$. Proceeding inductively, we eventually obtain a function $h_{k-\ell+1}\colon \binom{[n]}{\ell-1}\to \mb{R}$ satisfying $h_{k-\ell}(S_{k-\ell})=\sum_{S_{k-\ell+1}\subset S_{k-\ell}}h_{k-\ell+1}(S_{k-\ell+1})$ for every $S_{k-\ell}\in \binom{[n]}{\ell}$. Tracing back through the expansion, it follows that for all $R\in \binom{[n]}{k}$,
\[
f(R)=(k-\ell+1)!\sum_{S\subset R}h_{k-\ell+1}(S).
\]
Hence, $f\in Y_{k,\ell}$, so $X_{k,\ell}\subseteq Y_{k,\ell}$.
We conclude $X_{k,\ell}=Y_{k,\ell}$. Finally, the uniqueness of $h$ follows from Gottlieb's result on the rank of higher inclusion matrices (see \cref{lem:inclusion-matrix}).
\end{proof}

In the proof of \cref{prop:algebraic-criterion}, we used the following two results. The first is the base case $\ell=k$ of \cref{prop:algebraic-criterion}.

\begin{lemma}\label{lem:base-case}
Let $n,k\in \mb{N}$ with $n\ge 2k$. A function $f\colon \binom{[n]}{k} \to \mb{R}$ satisfies $W^k(f)=0$ if and only if there exists $h\colon \binom{[n]}{k-1}\to \mb{R}$ such that for all $R\in \binom{[n]}{k}$, we have $f(R)=\sum_{S\subset R}h(S)$.
\end{lemma}

The second is a result of Gottlieb~\cite{Got66} concerning the rank of higher inclusion matrices, which we state here as a lemma.

\begin{lemma}[\cite{Got66}]\label{lem:inclusion-matrix}
Let $V$ be a finite set and $r,s\in \mb{N}$ with $|V|\ge r\ge s$.  Consider the matrix $M^r_s$ with rows indexed by $\binom{V}{r}$ and columns indexed by $\binom{V}{s}$, where the entry at position $(R,S)$ is 1 if $S\subseteq R$ and 0 otherwise. Then, $\on{rank}(M^r_s) =\min \left\{\binom{n}{r},\binom{n}{s}\right\}$.   
\end{lemma}

The rest of this section is devoted to the proof of \cref{lem:base-case}. We begin with a useful notation.

\begin{definition}\label{defn:top-sets}
Let $n,r\in \mb{N}$ with $n\ge 2r$. Two $r$-permutations $\vec{A},\vec{B}$ of $[n]$ are said to be disjoint if no element appears in both $\vec{A}$ and $\vec{B}$. For disjoint $r$-permutations $\vec{A}=(a_1,\ldots,a_r)$ and $\vec{B}=(b_1,\ldots,b_r)$, we write $\vec{A}\prec \vec{B}$ if
\[
b_1<\cdots<b_r \enskip \text{and $a_i<b_i$ for every $i\in [r]$.}
\]
Let $\mc{B}_{n,r}$ denote the set of all $r$-permutations $\vec{B}$ such that $\vec{A}\prec \vec{B}$ for some $r$-permutation $\vec{A}$.     
\end{definition}

We are now ready to prove \cref{lem:base-case}.

\begin{proof}[Proof of \cref{lem:base-case}]
Using the notation from the proof of \cref{prop:algebraic-criterion}, we aim to show $X_{k,k}=Y_{k,k}$. From that proof, we know $Y_{k,k}\subseteq X_{k,k}$, and \cref{lem:inclusion-matrix} gives $\dim Y_{k,k}=\binom{n}{k-1}$.
Thus, the lemma follows if we show $\dim X_{k,k} \le \binom{n}{k-1}$, which we now proceed to prove.

Given $\vec{A}=(a_1,\ldots,a_k)$ and $\vec{B}=(b_1,\ldots,b_k)$ with $\vec{A}\prec \vec{B}$, define the function $\phi_{\vec{A},\vec{B}}\colon\binom{[n]}{k}\to \mb{R}$ by setting $\phi_{\vec{A},\vec{B}}(R) = 
(-1)^{|R \cap \{b_1, \ldots, b_k\}|}$ if $|R \cap \{a_i, b_i\}| = 1$ for every $i \in [k]$, and $\phi_{\vec{A},\vec{B}}(R) =0$ otherwise.
Note that $\langle f,\phi_{\vec{A},\vec{B}}\rangle=0$ for every $f\in X_{k,k}$. 

\begin{claim}[\cite{FG89}]\label{claim:independent-set}
 For each $\vec{B}\in \mc{B}_{n,k}$, let $\lambda(\vec{B})$ be any $k$-permutation satisfying $\lambda(\vec{B})\prec \vec{B}$. Then the set $\{\phi_{\lambda(\vec{B}),\vec{B}}\colon \vec{B}\in \mc{B}_{n,k}\}$ is linearly independent and has size $|\mc{B}_{n,k}|=\binom{n}{k}-\binom{n}{k-1}$. 
\end{claim}

\begin{poc}
For completeness, we include a proof here. We begin by showing $|\mc{B}_{n,k}|=\binom{n}{k}-\binom{n}{k-1}$. Encode each $k$-permutation $\vec{B}=(b_1,\ldots,b_k)$ of $[n]$ as a sequence $s_0, \ldots, s_n$ of $\pm 1$ values. Set $s_0 = 1$, and for $i \in [n]$, let $s_i =1$ if $i \notin \{b_1,\ldots,b_k\}$, and $s_i=-1$ if $i\in \{b_1,\ldots,b_k\}$. It is not difficult to see that $\vec{B}\in \mc{B}_{n,k}$ if and only if all the partial sums $s_0,s_0+s_1,\ldots,s_0+\cdots+s_n$ are positive. Moreover, each such sequence $s_0, \ldots, s_n$ contains $k$ entries equal to $-1$ and $n-k+1$ entries equal to 1. By Bertrand's ballot theorem, the number of such sequences is $\frac{n-2k+1}{n+1} \binom{n+1}{k}$. Thus,
$|\mc{B}_{n,k}|=\frac{n-2k+1}{n+1} \binom{n+1}{k}=\binom{n}{k}-\binom{n}{k-1}$. 

Next, we show that $\{\phi_{\vec{A},\vec{B}}\colon \vec{B}\in \mc{B}_{n,k}\}$ is a linearly independent set. Define a partial order on $\binom{[n]}{k}$ as follows: $\{a_1,\ldots,a_k\}\preceq \{b_1,\ldots,b_k\}$ if $a_1<\cdots<a_k$, $b_1<\cdots<b_k$, and $a_i\le b_i$ for every $i\in [k]$. Note that for any $k$-permutations $\vec{A}=(a_1,\ldots,a_k)$ and $\vec{B}=(b_1,\ldots,b_k)$ with $\vec{A}\prec \vec{B}$, if $\phi_{\vec{A},\vec{B}}(R)\ne 0$, then $R\preceq \{b_1,\ldots,b_k\}$. Consider the matrix representing the set $\{\phi_{\vec{A},\vec{B}}\colon \vec{B}\in \mc{B}_{n,k}\}$, where the columns are indexed by $\binom{[n]}{k}$ arranged in an order compatible with $\preceq$. This matrix is in echelon form, so it has full rank. Therefore, the set $\{\phi_{\vec{A},\vec{B}}\colon \vec{B}\in \mc{B}_{n,k}\}$ is linearly independent.
\end{poc}

Since $\langle f,\phi_{\vec{A},\vec{B}}\rangle=0$ for every $f\in X_{k,k}$, it follows from \cref{claim:independent-set} that
\[
\dim X_{k,k} \le \binom{n}{k}-\operatorname{rank}\{\phi_{\vec{A},\vec{B}}\colon \vec{A}\prec \vec{B}\} \le \binom{n}{k-1},
\] 
completing the proof of \cref{lem:base-case}.
\end{proof}

\subsection{The stability result}\label{sec:stability}
In this section, we prove \cref{prop:stability}. For integers $0\le k\le n$, let $\on{Slice}(n,k)$ denote the subset of $\{0,1\}^n$ that consists of vectors with exactly $k$ ones. Our argument involves extending a function defined on $\on{Slice}(n,k)$ to a function on $\mathbb{R}^n$. In this context, the appropriate extension is to interpret the function as a {\em harmonic multilinear polynomial}. 
Following Dunkl \cite{Dun76}, we say a multilinear polynomial $P\in \mb{R}[x_1,\ldots,x_n]$ is {\em harmonic} if
\[
\sum_{i=1}^n \frac{\partial P}{\partial x_i}=0.
\]
We denote by $\mc{H}'_{n,r}$ the vector space of harmonic multilinear polynomials in $x_1,\ldots,x_n$ of degree $r$.

We need a few key ingredients to prove \cref{prop:stability}. First, we require an explicit orthogonal basis for the space of harmonic multilinear polynomials, constructed by Srinivasan \cite{Sri11} and Filmus \cite{Fil16}. To describe this basis, we use some notation from \cite{Fil16}. The inner product of two functions $f, g\colon \on{Slice}(n,k) \to \mathbb{R}$ is defined by $\langle f, g \rangle = \underset{\vec{\sigma}}{\mb{E}} [f(\vec{\sigma})g(\vec{\sigma})]$, where $\vec{\sigma}\sim \on{Slice}(n,k)$. The norm of a function $f\colon \on{Slice}(n,k)\to \mb{R}$ is then given by $\|f\|= \sqrt{\langle f, f \rangle }$. Recall the notation from \cref{defn:top-sets}. If $\vec{A}=(a_1,\ldots,a_r)$ and $\vec{B}=(b_1,\ldots,b_r)$ are disjoint, define $\chi_{\vec{A},\vec{B}}=\prod_{i=1}^r(x_{a_i}-x_{b_i})$, a multilinear polynomial in  $\mb{R}[x_1,\ldots,x_n]$. For $\vec{B}\in \mc{B}_{n,r}$, define $\chi_{\vec{B}}=\sum\limits_{\vec{A}\prec \vec{B}}\chi_{\vec{A},\vec{B}}$, with $\chi_{\emptyset}=1$.
Finally, set $\chi_r=\prod_{i=1}^r(x_{2i-1}-x_{2i})$,
with the convention that $\chi_0=1$.

\begin{lemma}[{\cite[Theorems 3.1, 3.2 and 4.1]{Fil16}}] \label{lem:explicit-basis} The set $\{\chi_{\vec{B}}:\vec{B}\in \mc{B}_{n,r}\}$ is an orthogonal basis of $\mathcal{H}'_{n,r}$ with respect to the uniform measure on the slice $\on{Slice}(n,k)$. The subspaces $\mc{H}'_{n,r}$ for different $r$ are mutually orthogonal. For any $\vec{B}=(b_1,\ldots,b_r)\in \mc{B}_{n,r}$, we have $\|\chi_{\vec{B}}\|^2=c(\vec{B})\|\chi_r\|^2$, where $c(\vec{B})=\prod_{i=1}^r\binom{b_i-2i+2}{2}$ and $\|\chi_r\|^2=\frac{2^r\binom{n-2r}{k-r}}{\binom{n}{k}}$. 
\end{lemma}

The following is an immediate consequence of \cref{lem:explicit-basis}, which we will use.

\begin{corollary}\label{cor:orthogonal-basis}
Let $P\in \mb{R}[x_1,\ldots,x_n]$ be a harmonic multilinear polynomial of degree $d$. Then, with respect to the uniform measure on $\on{Slice}(n,k)$, $P$ admits an orthogonal decomposition $P=P^{=0}+\cdots+P^{=d}$, where $P^{=r}$ denotes the homogeneous degree-$r$ component of $P$. Moreover, for each $r$, we have 
\[
P^{=r}=\sum_{\vec{B}\in \mc{B}_{n,r}}\frac{\langle P, \chi_{\vec{B}} \rangle}{\|\chi_{\vec{B}}\|^2}\cdot \chi_{\vec{B}} \enskip \text{and} \enskip \|P^{=r}\|^2 = \sum_{\vec{B} \in \mathcal{B}_{n,r}} \frac{\langle P, \chi_{\vec{B}} \rangle^2}{c(\vec{B}) \|\chi_r\|^2},
\]
where $\|\chi_r\|^2=\frac{2^r\binom{n-2r}{k-r}}{\binom{n}{k}}$ and for each $\vec{B}=(b_1,\ldots,b_r)\in \mc{B}_{n,r}$, we define $c(\vec{B})=\prod_{i=1}^r\binom{b_i-2i+2}{2}$. 
\end{corollary}

We identify $\on{Slice}(n,k)$ with $\binom{[n]}{k}$ via the correspondence $(\sigma_1,\ldots,\sigma_n)\mapsto \{i\in [n]\colon \sigma_i=1\}$. This allows us to define the $W$-vector $(W^1(P),\ldots,W^k(P))$ of a polynomial $P\in \mb{R}[x_1,\ldots,x_n]$ as in \cref{defn:W-vector}, by viewing $P$ as a function on $\binom{[n]}{k}$. We will apply \cref{cor:orthogonal-basis} to bound $\|P^{=r}\|$ in terms of $W^r(P)$.

\begin{lemma}\label{lem:level-r}
Let $n,k,r \in \mb{N}$ with $n\ge 2k\ge 2r$. For any harmonic multilinear polynomial $P\in\mb{R}[x_1,\ldots,x_n]$, 
\[
\|P^{=r}\|^2\le 4^r \binom{k}{r} (W^r(P))^2.
\]
\end{lemma}

We first derive \cref{prop:stability} from \cref{lem:level-r} and then prove \cref{lem:level-r}.

\begin{proof}[Proof of \cref{prop:stability} assuming \cref{lem:level-r}] Without loss of generality, assume $V=[n]$. By \cite[Theorem 3.6]{FM19}, there exists a harmonic multilinear polynomial $P\in \mb{R}[x_1,\ldots,x_n]$ of degree at most $k$ that agrees with $f$ on $\binom{[n]}{k}\equiv \on{Slice}(n,k)$. Define $\tilde{f}=\sum_{r=0}^{\ell-1}P^{=r}$. From \cref{cor:orthogonal-basis,lem:level-r}, it follows that $\|f-\tilde{f}\|^2=\|P-\tilde{f}\|^2=\sum_{r=\ell}^k\|P^{=r}\|^2 \le 5^k\eps^2$, and thus $\|f-\tilde{f}\| \le 5^{k/2}\eps$.
In particular, for all but at most $5^{k/2}\eps^{1/2}\binom{n}{k}$ sets $R\in \binom{[n]}{k}$, we have $|f(R)-\tilde{f}(R)| \le \eps^{1/2}$.

By \cref{cor:orthogonal-basis}, $\tilde{f}$ is a linear combination of polynomials $\chi_{\vec{A},\vec{B}}$, where $\vec{A}\prec\vec{B}\in \bigcup_{r=0}^{\ell-1}\mc{B}_{n,r}$. For disjoint $r$-permutations $\vec{A}=(a_1,\ldots,a_r)$ and $\vec{B}=(b_1,\ldots,b_r)$, define the function $\lambda_{\vec{A},\vec{B}}\colon \binom{[n]}{\ell-1}\to \mb{R}$ by $\lambda_{\vec{A},\vec{B}}(S)=\binom{k-r}{\ell-1-r}^{-1} \cdot (-1)^{|S\cap \{b_1,\ldots,b_r\}|}$ if $|S\cap \{a_i,b_i\}|=1$ for each $i\in [r]$, and $\lambda_{\vec{A},\vec{B}}(S)=0$ otherwise. For $R\in \binom{[n]}{k}$, we have $\chi_{\vec{A},\vec{B}}(R)=\sum_{S\subset R,\ |S|=\ell-1}\lambda_{\vec{A},\vec{B}}(S)$. Therefore, there exists $h\colon \binom{[n]}{\ell-1}\to \mb{R}$ such that for all $R\in \binom{[n]}{k}$, we have $\tilde{f}(R)=\sum_{S\subset R, \ |S|=\ell-1}h(S)$. It then follows from ``the easy direction'' of \cref{prop:algebraic-criterion} that $W^k(\tilde{f})=\cdots=W^{\ell}(\tilde{f})=0$. 
\end{proof}

\begin{proof}[Proof of \cref{lem:level-r}]
From \cref{cor:orthogonal-basis}, we know
\[
\|P^{=r}\|^2=\sum_{\vec{B} \in \mathcal{B}_{n,r}} \frac{\langle P, \chi_{\vec{B}} \rangle^2}{c(\vec{B}) \|\chi_r\|^2}.
\]
For $\vec{B}=(b_1,\ldots,b_r)\in \mc{B}_{n,r}$, let $d(\vec{B})$ be the number of $r$-permutations $\vec{A}=(a_1,\ldots,a_r)$ with $\vec{A}\prec \vec{B}$. Given 
$a_1,\ldots,a_{i-1}$, there are at most $b_i - 2i+1$ choices for $a_i$. Thus, $d(\vec{B}) \leq \prod_{i=1}^r (b_i-2i+1)\le \sqrt{2^r c(\vec{B})}$.    
Moreover, by the Cauchy-Schwarz inequality, 
$\langle P,\chi_{\vec{B}} \rangle^2=\left(\sum_{\vec{A}\prec \vec{B}}\langle P,\chi_{\vec{A},\vec{B}} \rangle\right)^2 \le d(\vec{B}) \sum_{\vec{A} \prec \vec{B}} \langle P,\chi_{\vec{A},\vec{B}} \rangle^2$.  
Therefore,
\[
\|P^{=r}\|^2 \le \sum_{\vec{B}\in \mc{B}_{n,r}}\frac{d(\vec{B})}{c(\vec{B}) \|\chi_r \|^2}\sum_{\vec{A}\prec \vec{B}}\langle P,\chi_{\vec{A},\vec{B}} \rangle^2
\le \frac{2^r}{\|\chi_r\|^2}\sum_{\vec{B}\in \mc{B}_{n,r}}\frac{1}{d(\vec{B})}\sum_{\vec{A}\prec \vec{B}}\langle P,\chi_{\vec{A},\vec{B}} \rangle^2.
\]
For any permutation $\pi\in \on{Sym}(n)$, applying $\pi$ to the ground set $[n]$ and using the inequality above yields
\[
\|P^{=r}\|^2  \le \frac{2^r}{\|\chi_r\|^2}\sum_{\vec{B}\in \mc{B}_{n,r}}\frac{1}{d(\vec{B})}\sum_{\vec{A}\prec \vec{B}}\langle P,\chi_{\pi(\vec{A}),\pi(\vec{B})} \rangle^2.
\]
Averaging over all $\pi \in \on{Sym}(n)$, we get 
\[
\|P^{=r}\|^2  \le \frac{2^r}{\|\chi_r\|^2}\sum_{\vec{B}\in \mc{B}_{n,r}}\frac{1}{d(\vec{B})}\sum_{\vec{A}\prec \vec{B}}\underset{\pi}{\mb{E}}\langle P,\chi_{\pi(\vec{A}),\pi(\vec{B})} \rangle^2.
\]
We observe that, for a fixed pair of $r$-permutations $\vec{A},\vec{B}$ satisfying $\vec{A}\prec \vec{B}$, the pair $\pi(\vec{A}),\pi(\vec{B})$ forms a uniform random $2r$-permutation of $[n]$ when $\pi\sim \on{Sym}(n)$. Hence $\mathbb{E}_{\pi}\langle P,\chi_{\pi(\vec{A}),\pi(\vec{B})}\rangle^2=\mb{E}_{\vec{C},\vec{D}}\langle P,
\chi_{\vec{C},\vec{D}}\rangle^2$, where $(\vec{C},\vec{D})$ is a uniformly random $2r$-permutation of $[n]$, so that $\frac{1}{d(\vec{B})}\sum_{\vec{A}\prec \vec{B}}\mb{E}_{\pi}\langle P,\chi_{\pi(\vec{A}),\pi(\vec{B})} \rangle^2=\mb{E}_{\vec{C},\vec{D}}\langle P,\chi_{\vec{C},\vec{D}} \rangle^2$. Moreover, it follows from the definitions of $\langle \cdot,\cdot\rangle$ and $W^r(\cdot)$ that $\mb{E}_{\vec{C},\vec{D}}\langle P,\chi_{\vec{C},\vec{D}} \rangle^2=\left(\frac{2^r\binom{n-2r}{k-r}}{\binom{n}{k}}\right)^2(W^r(P))^2$.
Therefore, the inequality above is simplified to
\[
\|P^{=r}\|^2 \le\frac{2^r}{\|\chi_r\|^2}\cdot|\mc{B}_{n,r}| \cdot \left(\frac{2^r\binom{n-2r}{k-r}}{\binom{n}{k}}\right)^2(W^r(P))^2.
\]
Substituting $\|\chi_r\|^2=\frac{2^r\binom{n-2r}{k-r}}{\binom{n}{k}}$ and $|\mc{B}_{n,r}|\le \binom{n}{r}$ into the inequality, we get $\|P^{=r}\|^2 \le 4^r\binom{k}{r}(W^r(P))^2$.
\end{proof}
We conclude this section with an $L_2$-version of a result by Bollob\'as and Scott \cite[Lemma 9]{BS15}, which may be of independent interest, although it is not required for the purposes of this paper.

\begin{corollary}\label{cor:BS-L2}
Let $n,k\in \mb{N}$ with $n\ge 2k$. For any function $f\colon \binom{[n]}{k}\rightarrow \mb{R}$,
\[
\max_{1\le r\le k} W^r(f)\ge 5^{-k/2}\cdot \sqrt{\on{Var}[f]}.
\]
\end{corollary}
\begin{proof}
Let $P\in \mb{R}[x_1,\ldots,x_n]$ be the harmonic multilinear polynomial of degree at most $k$ that agrees with $f$ on $\on{Slice}(n,k)$. By \cref{cor:orthogonal-basis} and \cref{lem:level-r}, we obtain $P^{=0}=\mb{E}P$ and
\[
\on{Var}[P]=\sum_{r=1}^k\|P^{=r}\|^2\le \sum_{r=1}^k4^r \binom{k}{r} (W^r(P))^2\le 5^k \cdot \max_{1\le r \le k} (W^r(P))^2.
\]
Thus, $\max_{1\le r\le k} W^r(f)=\max_{1\le r \le k} W^r(P)\ge 5^{-k/2}\cdot\sqrt{\on{Var}[P]}=5^{-k/2}\cdot\sqrt{\on{Var}[f]}$.
\end{proof}

\section{Unavoidable patterns}\label{sec:robust-unavoidable}

The hypergraph Ramsey theorem guarantees that any sufficiently large $k$-uniform hypergraph contains a large homogeneous subhypergraph. In general, no structure beyond homogeneous subhypergraphs can be guaranteed, since the entire host hypergraph may be homogeneous. However, Fox and Sudakov \cite{FS08} show that when both edges and non-edges are well represented, a richer family of patterns necessarily emerges.

\begin{definition}\label{defn:pattern}
Let $k,t,s\in \mb{N}$ with $k,t\ge 2$. A $(k,t,s)$-pattern is a $k$-uniform hypergraph $F$ admitting a partition $V(F)=V_1 \cup \cdots \cup V_t$, where each part $V_i$ has size $s$, such that for every $S\in \binom{V(F)}{k}$, the indicator $\mathbf{1}_{F}(S)$ depends only on
$(|S \cap V_1|, \dots, |S \cap V_t|)$. We say $F$ is homogeneous if it is either empty or complete.
\end{definition}

The main result of this section is a robust version of Fox--Sudakov's theorem.

\begin{proposition}\label{prop:robust-unavoidable}
For $k,s\in \mb{N}$ and $\delta>0$, there exists $\gamma=\gamma_{\ref{prop:robust-unavoidable}}(k,s,\delta)>0$ such that the following holds for sufficiently large $n$. Let $G$ be a $k$-uniform hypergraph on $n$ vertices with $\min\{e(G),\binom{n}{k}-e(G)\}\ge \delta n^k$, and let $W\subseteq \binom{V(G)}{k}$ satisfy $|W| \le \gamma n^k$. Then $G$ contains a non-homogeneous $(k,k,s)$-pattern as an induced subhypergraph that avoids $W$.     
\end{proposition}

To prove \cref{prop:robust-unavoidable}, we rely on several key ingredients. The first is Fox--Sudakov's theorem \cite[Theorem 4.2]{FS08}, stated below as a lemma. A complete proof can be found in \cite[Theorem A.1]{KST19}.

\begin{lemma}[\cite{FS08}]\label{lem:unavoidable}
    For $k,s\in \mb{N}$ and $\delta>0$, there exists $N=N_{\ref{lem:unavoidable}}(k,s,\delta)\in \mb{N}$ such that the following holds. Suppose $G$ is a $k$-uniform hypergraph on $n\ge N$ vertices with $\min\{e(G),\binom{n}{k}-e(G)\}\ge \delta n^k$. Then $G$ contains a non-homogeneous $(k,k,s)$-pattern as an induced subhypergraph. 
\end{lemma}

We also need the following result due to Kwan, Sudakov, and Tran \cite[Lemma A.3]{KST19}.

\begin{lemma}[\cite{KST19}]\label{lem:alternating-path}
For any $k\in \mb{N}$ and $\delta>0$, there exists $\alpha=\alpha_{\ref{lem:alternating-path}}(k,\delta)>0$ such that the following holds. Consider a partial red-blue colouring of the complete $k$-uniform hypergraph on $n$ vertices, where each colour class contains at least $\delta n^k$ edges and at most  $\alpha n^k$ edges are left uncoloured. Then there are at least $\alpha n^{k-1}$ $(k-1)$-subsets, each contained in at least $\alpha n$ red edges and $\alpha n$ blue edges.    
\end{lemma}

Finally, we deduce \cref{prop:robust-unavoidable} from \cref{lem:unavoidable,lem:alternating-path} via a probabilistic argument.

\begin{proof}[Proof of \cref{prop:robust-unavoidable}]
Let $k, s \in \mathbb{N}$ and $\delta > 0$ be fixed. Set
\[
\alpha=\alpha_{\ref{lem:alternating-path}}(k,\delta/2), m= \max\left\{N_{\ref{lem:unavoidable}}(k,s,\alpha^3/k2^{k+2}), 2k\right\}, \enskip \text{and} \enskip \gamma = \min\left\{\alpha^3/16m^k,\delta/2\right\}.
\]
We refer to the elements of $W$ as white edges. 

Pick a set $M$ of $m$ vertices uniformly at random from $V(G)$. By the union bound, the probability that $M$ contains a white edge is at most
\[
\binom{m}{k}\cdot\frac{\gamma n^k}{\binom{n}{k}}\le \left(\frac{m}{n}\right)^k\cdot \gamma n^k= m^k\gamma.
\]

Let $X$ denote the random variable counting the number of pairs $(S,T)$ such that $S$ is an edge of $G[M]$, $T$ is a non-edge of $G[M]$, and $|S\cap T|=k-1$. For a given pair $S,T\in \binom{V(G)}{k}$ with $|S\cap T|=k-1$, the probability that $S\cup T$ is a subset of $M$ is at least $\frac{\binom{n-k-1}{m-k-1}}{\binom{n}{m}} \ge (m/2n)^{k+1}$ given that $n\ge m\ge 2k$. Moreover, by \cref{lem:alternating-path}, there are $\alpha^3 n^{k+1}$ pairs $S,T$ such that $S$ is an edge of $G$, $T$ is a non-edge of $G$, and $|S\cap T|=k-1$. Therefore, by the linearity of expectation, we have
\[
\mb{E} X \ge \alpha^3 n^{k+1} \cdot (m/2n)^{k+1}=(\alpha^3/2^{k+1}) m^{k+1}.
\]
Since $X\le m^2\binom{m}{k-1} < 2^{-k+2}m^{k+1}$, it follows that with probability greater than $\alpha^3/16$, we have 
\[
X\ge (\alpha^3/2^{k+2}) m^{k+1}.
\] 

Given that $m^k\gamma \le \alpha^3/16$, there exists a choice of $M$ for which $X\ge (\alpha^3/2^{k+2}) m^{k+1}$ and $M$ contains no white edges. Consequently, both the number of edges and the number of non-edges in the induced subhypergraph $G[M]$ are at least $X/(km)\ge (\alpha^3/k2^{k+2})m^k$. Applying \cref{lem:unavoidable} to $G[M]$, we conclude that $G[M]$ contains a non-homogeneous $(k,k,s)$-pattern.
\end{proof}

\section{Completion of the proof}\label{sec:proofs}

In this section, we prove \cref{thm:main}, establishing the lower bound $\on{bs}(k)\ge 3$ in \cref{sec:lower-bound} and the upper bound $\on{bs}(k)\le g(k)+2$ in \cref{sec:upper-bound}.

\subsection{Pairs with zero discrepancy}\label{sec:lower-bound}

The lower bound in \cref{thm:main} follows directly from the following.

\begin{proposition}\label{prop:counterexample} For any integer $k\ge 3$, there exist constants $\gamma=\gamma(k)\in (0,1/2)$ and $n_0=n_0(k)\in \mb{N}$ such that for all $n\ge n_0$, one can construct $k$-uniform hypergraphs $G$ and $H$ on $n$ vertices with densities in $(\gamma,1-\gamma)$, satisfying $\on{disc}(G,H)=0$. 
\end{proposition}

\begin{proof}
An $(n, k, r, \lambda)$-block design is a $k$-uniform hypergraph on $[n]$ in which every set of $r$ vertices is contained in exactly $\lambda$ edges. Glock, Kühn, Lo and Osthus \cite[Theorem 1.1 and Corollary 4.14]{GKLO21} proved that for all $k, r \in \mathbb{N}$ with $k>r$, there exist $n_0 \in \mb{N}$ and $\eps > 0$ such that for all $n \geq n_0$ and $\lambda \in \mathbb{N}$ satisfying $\lambda \leq \eps n$ and 
\[
\binom{k-i}{r-i}\big| \lambda \binom{n-i}{k-i} \quad \text{for all } 0 \leq i < r,
\]
an $(n, k, r, \lambda)$-block design exists. We apply this result with $r=k-1$ and choose $\lambda$ to be the largest multiple of $\prod_{i=0}^{r-1} \binom{k-i}{r-i}$ not exceeding $\varepsilon n$. This guarantees that for some integer $\lambda \in \left[ \eps n/2, \eps n \right]$, an $(n, k, k-1, \lambda)$-block design exists.
Let $G$ denote this $k$-uniform hypergraph. 

Partition $[n]$ into $k-1$ disjoint parts $V_1,\ldots,V_{k-1}$ of sizes differing by at most one. Define the $k$-uniform hypergraph $H$ on the vertex set $[n]$ by including a set $R\in \binom{[n]}{k}$ in $E(H)$ if $R$ intersects every part $V_i$. We see that both $G$ and $H$ have moderate densities. Additionally, for every pair $G',H'$ with $G'\cong G$, $H'\cong H$ and $V(G')=V(H')$, we have
$|E(G')\cap E(H')|= \frac{1}{2}\lambda|V_1|\cdots |V_{k-1}|$. Hence $\on{disc}(G,H)=0$.
\end{proof}

\subsection{Proof of Theorem~\ref{thm:main}}\label{sec:upper-bound}

In this section, we combine the tools developed in \cref{sec:W-vectors,sec:robust-unavoidable} to prove the upper bound $\on{bs}(k) \le g(k) + 2$. 
Our argument relies on two key ingredients. The first is a reduction lemma that reduces the problem of proving the upper bound to analysing the $W$-vector of unavoidable patterns.

\begin{lemma}\label{lem:reduction}
For every pair of integers $k\ge \ell$ and every $\delta>0$, there exists $\eps=\eps_{\ref{lem:reduction}}(k,\ell,\delta)>0$ such that the following holds for sufficiently large $n$. Suppose any non-homogeneous $(k,k,k)$-pattern $F$ satisfies 
\[
\max_{\ell\le r \le k}W^r(F)>0.
\]
Then, for any $k$-uniform hypergraph $G$ on the vertex set $[n]$ with $\min\{e(G),\binom{n}{k}-e(G)\}\ge \delta n^k$, we have 
\[
\max_{\ell\le r \le k}W^r(G)\ge \eps.
\]
Consequently, $\on{bs}(k) \le k-\ell+2$.   
\end{lemma}

The second result verifies the assumption on the $W$-vector of unavoidable patterns in \cref{lem:reduction}.

\begin{lemma}\label{lem:pattern}
Let $k,t\ge 2$ and $F$ be a $(k,t,k)$-pattern with $W^k(F)=\cdots=W^{k-g(k)}(F)=0$.  Then $F$ is homogeneous.  
\end{lemma}

\begin{proof}[Proof of \cref{thm:main} assuming \cref{lem:reduction,lem:pattern}]
By \cref{lem:pattern}, any non-homogeneous $(k,k,k)$-pattern $F$ satisfies $\max_{\ell\le r \le k}W^r(F)>0$.  Applying \cref{lem:reduction} with $\ell=k-g(k)$ yields $\on{bs}(k)\le g(k)+2$. 
\end{proof}

In the rest of this section, we prove \cref{lem:reduction,lem:pattern}. We start with \cref{lem:reduction}.

\begin{proof}[Proof \cref{lem:reduction}]
Let $\gamma=\gamma_{\ref{prop:robust-unavoidable}}(k,k,\delta)>0$, and choose $\eps>0$ sufficiently small with respect to $k$ and $\ell$. Suppose, for contradiction, that $\max_{\ell\le r \le k}W^r(G) \le \eps$. By \cref{prop:stability}, there exists a function $\tilde{f}\colon \binom{[n]}{k}\to \mb{R}$ with the following properties:
\begin{itemize}
    \item[\rm (i)] For all but at most $5^{k/2}\eps^{1/2}\binom{n}{k}$ sets $R\in \binom{[n]}{k}$, we have $|\mbf{1}_G(R)-\tilde{f}(R)| \le \eps^{1/2}$;
    \item[\rm (ii)] $W^k(\tilde{f})=\cdots=W^{\ell}(\tilde{f})=0$.
\end{itemize}   
Let $W$ be the set of $R\in \binom{[n]}{k}$ with $|\mbf{1}_G(R)-\tilde{f}(R)| > \eps^{1/2}$. Property (i) implies $|W|\le 5^{k/2}\eps^{1/2}\binom{n}{k} \le \gamma n^k$. Applying \cref{prop:robust-unavoidable} to $G$ and $W$, we find that $G$ contains a non-homogeneous $(k,k,k)$-pattern $F$ as an induced subhypergraph such that $E(F)\cap W=\emptyset$. Since $F$ avoids $W$, it follows from the definition of $W$ that $|\mbf{1}_F(R)-\tilde{f}(R)| \le \eps^{1/2}$ for all $R\in \binom{V(F)}{k}$. Let $\bar{f}$ denote the restriction of $\tilde{f}$ to $\binom{V(F)}{k}$. Then $\|\mbf{1}_F-\bar{f}\|_{\infty} \le \eps^{1/2}$. By property (ii) and \cref{prop:algebraic-criterion}, we have $W^k(\bar{f})=\cdots=W^{\ell}(\bar{f})=0$. Hence, by \cref{fact},
\[
\max_{\ell\le r\le k}W^r(F)=\max_{\ell\le r\le k}W^r(\mbf{1}_F)\le \max_{\ell\le r\le k}W^r(\bar{f})+\|\mbf{1}_F-\bar{f}\|_{\infty}\le \eps^{1/2}.
\]
For sufficiently small $\eps$, this implies $\max_{\ell\le r\le k}W^r(F)=0$, contradicting the assumption. Therefore, $\max_{\ell\le r \le k}W^r(G)\ge \eps$.

To prove the second part, consider any collection of $n$-vertex $k$-uniform hypergraphs $G_1,\ldots,G_{k-\ell+2}$, each satisfying $\min\{e(G_i),\binom{n}{k}-e(G_i)\}\ge \delta n^k$. By the first part of the lemma, for each $i\in [k-\ell+2]$, 
there exists some $r(i)\in [\ell,k]$ such that $W^{r(i)}(G_i)\ge \eps$. The pigeonhole principle then guarantees a pair $i<j$ with $r(i)=r(j)$; denote this common value by $r$. Then \cref{lem:Bollobas-Scott} implies 
\[
\on{disc}(G_i,G_j)\ge c_{\ref{lem:Bollobas-Scott}}n^{(k+1)/2} (W^r(G_i)W^r(G_j))^2=\Omega_k(n^{(k+1)/2}),
\]
which shows $\on{bs}(k)\le k-\ell+2$.
\end{proof}

We will prove \cref{lem:pattern} by induction on $t$, the number of parts in an unavoidable pattern. The base case $t=2$ corresponds to the following lemma.

\begin{lemma}\label{lem:bipartite-pattern}
Let $k\ge 2$ be an integer and $F$ be a $(k,2,k)$-pattern with $W^k(F)=\cdots=W^{k-g(k)}(F)=0$. Then $F$ is homogeneous.
\end{lemma}
\begin{proof} Let $V=V_1\cup V_2$ be the vertex partition of $F$, and let $\ell=k-g(k)$. 

\begin{claim}\label{claim:pattern-function}
There exists $h \colon \binom{V}{\ell-1} \to \mb{R}$ such that for every $R \in \binom{V}{k}$, we have
$\mbf{1}_F(R) = \sum_{S \subset R} h(S)$,
with $h(S)$ depending only on $(|S\cap V_1|,|S\cap V_2|)$.  \end{claim}
\begin{poc}
By \cref{prop:algebraic-criterion}, there exists a unique function $h \colon \binom{V}{\ell-1} \to \mb{R}$ such that for all $R \in \binom{V}{k}$, we have $\mbf{1}_F(R) = \sum_{S \subset R} h(S)$.
Consider a permutation $\pi$ of $V$ that preserves each part $V_i$. Let $h_{\pi}$ be the function on $\binom{V}{\ell-1}$ given by $h_{\pi}(S)=h(\pi(S))$. For every $R\in \binom{V}{k}$, we have 
\[
\sum_{S \subset R} h_{\pi}(S)=\sum_{S \subset R} h(\pi(S))=\mbf{1}_{F}(\pi(R))=\mbf{1}_F(R).
\]
By the uniqueness of $h$, we obtain $h_{\pi}=h$ for any such $\pi$. Hence $h(S)$ depends only on $(|S\cap V_1|,|S\cap V_2|)$.    
\end{poc}

Let $\vec{\alpha}=(\alpha_0,\ldots,\alpha_k)$ denote a 0-1 vector where $\alpha_i=1$ if $R\in E(F)$ for every set $R\in \binom{V}{k}$ with $(|R\cap V_1|,|R\cap V_2|)=(i,k-i)$, and $\alpha_i=0$ otherwise.

\begin{claim}\label{claim:linear-system} For every integer $\ell \le r\le k$, we have $\sum_{i=0}^r(-1)^i\binom{r}{i}\alpha_i=0$.  
\end{claim}
\begin{poc} Let $h\colon \binom{V(F)}{\ell-1}\to \mb{R}$ be the function given by \cref{claim:pattern-function}. For $0\le j\le \ell-1$, define $\beta_j=h(S)$, where $S$ is any set with $(|S\cap V_1|,|S\cap V_2|)=(j,\ell-1-j)$. We then have $\alpha_i=\sum_{j}\binom{i}{j}\binom{k-i}{\ell-1-j}\beta_j$. Substituting this expression into the sum, we get \[ \sum_{i=0}^r(-1)^i\binom{r}{i}\alpha_i=\sum_j \left( \sum_i (-1)^i\binom{r}{i} \binom{i}{j}\binom{k-i}{\ell-1-j}\right)\beta_j. \] Using the identity $\binom{r}{i}\binom{i}{j}=\binom{r}{j}\binom{r-j}{i-j}$, we rewrite the sum as \[ \sum_j \left(\sum_{i}(-1)^{i-j}\binom{r-j}{i-j}\binom{k-i}{\ell-1-j}\right) (-1)^j\binom{r}{j}\beta_j.\]By a change of variables $i^*=i-j$, $r^*=r-j$, $k^*=k-j$, and $\ell^*=\ell-j$, we transform the inner sum to \[ \sum_{i}(-1)^{i-j}\binom{r-j}{i-j}\binom{k-i}{\ell-1-j}=\sum_{i^*}(-1)^{i^*}\binom{r^*}{i^*}\binom{k^*-i^*}{\ell^*-1}. \] Notice that $\binom{r^*}{i^*}$ is the coefficient of $x^{i^*}$ in $(1+x)^{r^*}$, and $(-1)^{i^*}\binom{k^*-i^*}{\ell^*-1}$ is the coefficients of $x^{k^*-\ell^*+1-i^*}$ in $\frac{(-1)^{k^*-\ell^*+1}}{(1+x)^{\ell^*}}$. Thus, the inner sum corresponds to the coefficient of $x^{k^*-\ell^*+1}$ in $(-1)^{k^*-\ell^*+1}(1+x)^{r^*-\ell^*}$, which is 0 since $k^*-\ell^*+1>r^*-\ell^* \ge 0$. Therefore, we conclude $\sum_{i=0}^r(-1)^i\binom{r}{i}\alpha_i=0$. \end{poc} 

Recalling $g(k)$ from \cref{defn:gap}, we apply \cref{claim:linear-system} to conclude either $\vec{\alpha}=(0,\ldots,0)$ or $\vec{\alpha}=(1,\ldots,1)$. Hence the pattern $F$ is homogeneous. 
\end{proof}

We conclude this section with the proof of \cref{lem:pattern}.

\begin{proof}[Proof of \cref{lem:pattern}]
We prove \cref{lem:pattern} by induction on $t$. For the base case $t=2$, the conclusion follows from \cref{lem:bipartite-pattern}. Now assume $t\ge 3$, and suppose the statement holds for all $2\le t'<t$.

Let $\ell=k-g(k)$, and let $V=V_1\cup \cdots\cup V_t$ be the vertex partition of $F$. Since $W^k(F)=\cdots=W^{\ell}(F)=0$, by \cref{prop:algebraic-criterion}, there exists a function $h\colon\binom{V}{\ell-1}\to\mathbb{R}$ such that for every $R \in\binom{V}{k}$,
\begin{equation}\label{eq:combinatorics-algebra}
 \mbf{1}_F(R)=\sum_{S\subset R}h(S).   
\end{equation} 
If $h$ is identically zero, then \cref{eq:combinatorics-algebra} implies $F$ is empty, and we are done. Otherwise, there exists an integer $\ell_0 \ge 0$ such that $h(S)=0$ whenever $\min_{1\le i \le t} |S\cap V_i|<\ell_0$, but $h(S_0)\ne 0$ for some $S_0\in \binom{V}{\ell-1}$ with $\min_{1\le i\le t}|S_0\cap V_i|=\ell_0$. Without loss of generality, assume $|S_0\cap V_1|=\ell_0$, and let $I=S_0\cap V_1$.

For each $i\in [t]$, the induced subhypergraph $F[V\setminus V_i]$ is a $(k,t-1,k)$-pattern. Recalling \cref{rmk:uniqueness-induced}, we get $W^k(F[V\setminus V_i])=\cdots=W^{\ell}(F[V\setminus V_i])=0$. By the induction hypothesis, each $F[V \setminus V_i]$ is homogeneous, and hence, the $F[V \setminus V_i]$ are either all empty or all complete. Taking the complement of $F$ if necessary, we may assume they are all empty. 

Since $F[V\setminus V_i]$ is empty, it follows from \cref{rmk:uniqueness-induced} that the restriction of $h$ to each $\binom{V\setminus V_i}{\ell-1}$ is identically zero, which forces $\ell_0>0$. Consequently, $R^*\cup I \notin E(F)$ for every $R^*\in \binom{V_2}{k-\ell_0}$.

For $2\le i \le t$, choose a subset $U_i\subseteq V_i$ of size $k-\ell_0$. Let $F^*$ be the $(k-\ell_0,t-1,k-\ell_0)$-pattern with vertex set $U=U_2\cup \cdots\cup U_t$, where $R^* \in \binom{U}{k-\ell_0}$ is an edge of $F^*$ if only if $R^*\cup I \in E(F)$. 
As $R^*\cup I \notin E(F)$ for every $R^*\in \binom{U_2}{k-\ell_0}$, we have $R^*\notin E(F^*)$ for all such $R^*$.

Now define a function $h^*\colon \binom{U}{\ell-\ell_0-1}\to \mb{R}$ by setting $h^*(S^*)=h(S^*\cup I)$. We then have
\begin{align*}
\mbf{1}_{F^*}(R^*)=\mbf{1}_F(R^*\cup I)&\overset{\eqref{eq:combinatorics-algebra}}{=}\sum_{S\subset R^*\cup I} h(S)\\
&=\sum_{S\cap V_1=I}h(S)=\sum_{S^*\subset R^*}h^*(S^*),
\end{align*}
where in the second line we use the fact that $h(S)=0$ whenever $|S\cap V_1|<|I|=\ell_0$. By \cref{prop:algebraic-criterion}, $W^{k-\ell_0}(F^*)=\cdots=W^{\ell-\ell_0}(F^*)=0$. 

Since $g$ is non-decreasing, we have $\ell-\ell_0=k-g(k)-\ell_0\le (k-\ell_0)-g(k-\ell_0)$. By the induction hypothesis, $F^*$ is homogeneous. Since $R^*\notin E(F^*)$ for every $R^*\in \binom{U_2}{k-\ell_0}$, it follows that $F^*$ must be empty. By \cref{rmk:uniqueness-induced}, we conclude $h^*\equiv 0$. Therefore, we have $0=h^*(S_0\setminus V_1)=h(S_0)\ne 0$, a contradiction. 
\end{proof}

\bibliographystyle{plain} 
\bibliography{references} 

\appendix
\section{A number-theoretic result}

\begin{proof}[Proof of part (ii) of \cref{thm:Gathen-Roche}] Let $g=g(k)$. Fix $m\in [2,k]$, and consider the following system of linear equations
\[
\sum_{0\le i\le r}(-1)^i\binom{r}{i}\alpha_i=0 \quad \text{for all } r=m-g,\ldots,m,
\]
where $\vec{\alpha}=(\alpha_0,\ldots,\alpha_{m})\in \{0,1\}^{m+1}$. From \cref{defn:gap}, there exists a prime $p$ such that $m-g\le p-1\le m$. For $0\le i \le p-1$, we have $\binom{p-1}{i}=\frac{(p-1)(p-2)\cdots (p-i)}{i!}\equiv (-1)^i \mod{p}$. Thus, 
\[
0=\sum_{i=0}^{p-1}(-1)^i\binom{p-1}{i}\alpha_i\equiv \sum_{i=0}^{p-1}\alpha_i \mod{p}.
\]
Since $\alpha_i\in \{0,1\}$ for all $0\le i\le p-1$, this forces $\alpha_0=\cdots=\alpha_{p-1}$. We then solve $\alpha_p,\ldots,\alpha_m$ sequentially using the equations $\sum_{0\le i\le r}(-1)^i\binom{r}{i}\alpha_i=0$ for each $r=p,\ldots,m$. We conclude $\alpha_0=\cdots=\alpha_m$, which means the system has only two solutions $\vec{\alpha}=(0,\ldots,0)$ and $\vec{\alpha}=(1,\ldots,1)$. 
\end{proof}

\end{document}